\documentclass[11pt]{article}
\usepackage[utf8]{inputenc}
\usepackage[T1]{fontenc}
\usepackage[english]{babel}
\usepackage{amssymb}
\usepackage{amsmath}
\usepackage{amsthm}
\usepackage{amsfonts}
\usepackage{esint}
\usepackage{dsfont}
\usepackage{wasysym}
\usepackage{bbm}
\usepackage[all]{xy}
\usepackage{hyperref}
\usepackage{caption}
\usepackage{enumitem}
\usepackage{graphicx}
\usepackage{tikz-cd}
\usepackage{tkz-tab}
\usetikzlibrary{decorations.markings,decorations.pathreplacing,positioning,cd}
\usepackage{graphicx}
\usepackage{subcaption}

\usepackage{stmaryrd}

\newtheorem{theorem}{Theorem}[section]
\newtheorem{corollary}[theorem]{Corollary}
\newtheorem{assumption}[theorem]{Assumption}
\newtheorem{lemma}[theorem]{Lemma}
\newtheorem{claim}[theorem]{Fact}
\theoremstyle{definition}
\newtheorem{definition}[theorem]{Definition}

\newtheorem{proposition}[theorem]{Proposition}

\theoremstyle{remark}
\newtheorem{remark}[theorem]{Remark}

\newcommand{\R}{\mathbb{R}}
\newcommand{\Z}{\mathbb{Z}}
\newcommand{\N}{\mathbb{N}}

\newcommand{\Proba}[1]{\mathbb{P}\left( #1 \right)}

\title{Shadow and percolation I: discrete landscapes with independence}

\author{David Vernotte}

\begin{document}
\maketitle

\begin{abstract}
    Let $X : \Z^2 \to \R$ be a random field which we interpret as a random height function describing some landscape of mountains. We consider a source of light (a sun) located at infinity in a direction parallel with an axis of $\Z^2$ and emitting rays which are all parallel and make a slope $\ell\in \R$ with the horizontal plane. Given the value of $\ell\in \R$ some mountains of the landscape will be lit by the sun and other will be in the shadow of some higher mountain. Under some assumptions on $X$, including an independence assumption, we prove that this model may present two distincts phases depending on $\ell$. When the slope $\ell>0$ is very small, almost surely there exists an unbounded cluster of points in the shadow. However, if $\ell$ is big enough, almost surely there exists an unbounded cluster of points lit by the sun. We reformulate this problem in terms of percolation of a field $\alpha : \Z^2 \to \overline{\R}$ which has a simple definition but does not present many of the nice properties usually found in percolation models such as the FKG inequality, invariance by rotation or finite range correlations.
\end{abstract}

\tableofcontents

\section{Introduction}

\begin{figure}[htbp]
  \centering
  \begin{subfigure}{0.45\textwidth}
    \centering
    \includegraphics[width=\linewidth]{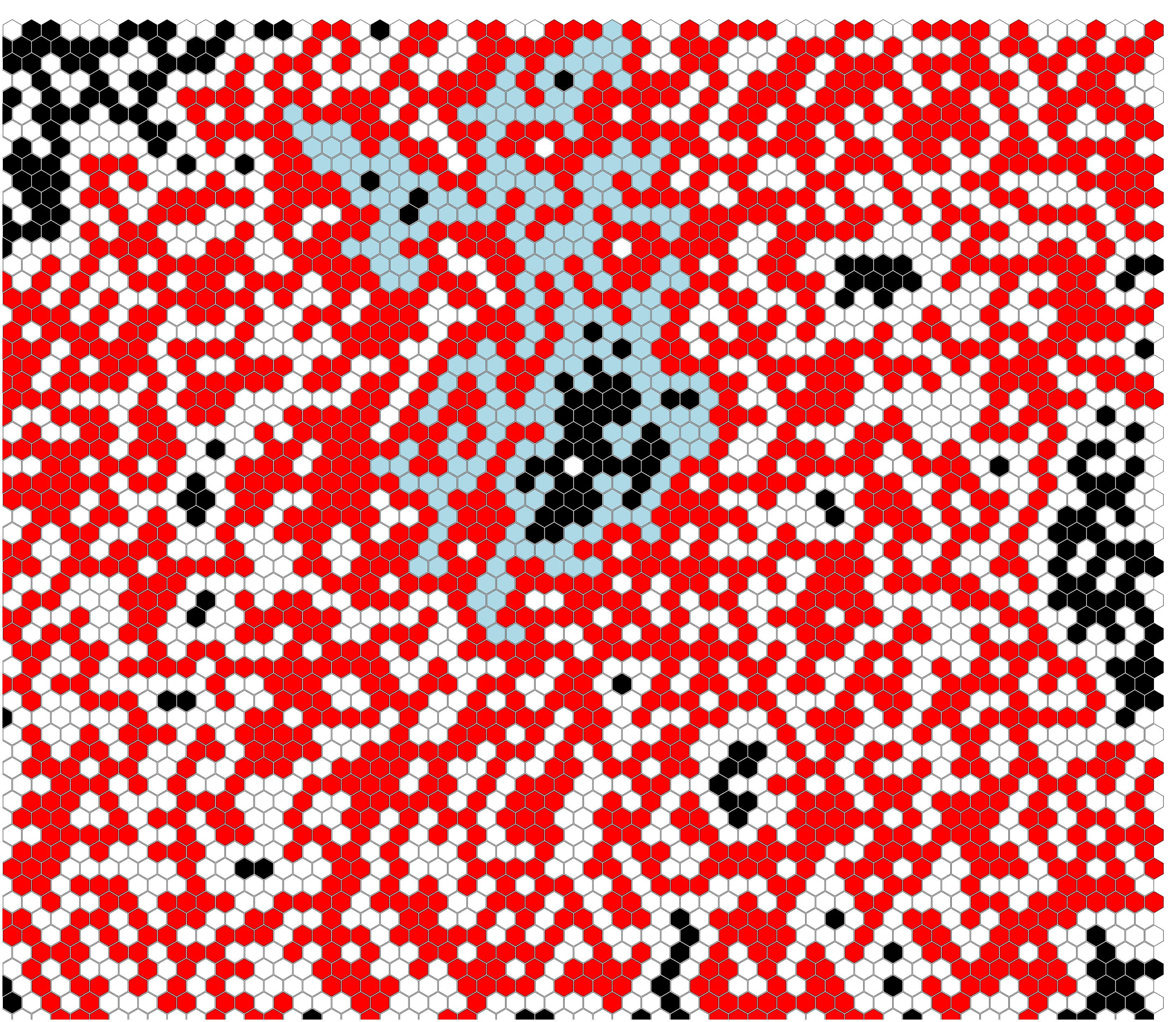}
    \caption{$\ell=0.4$}
    \label{fig:sub3}
  \end{subfigure}\hfill
  \begin{subfigure}{0.45\textwidth}
    \centering
    \includegraphics[width=\linewidth]{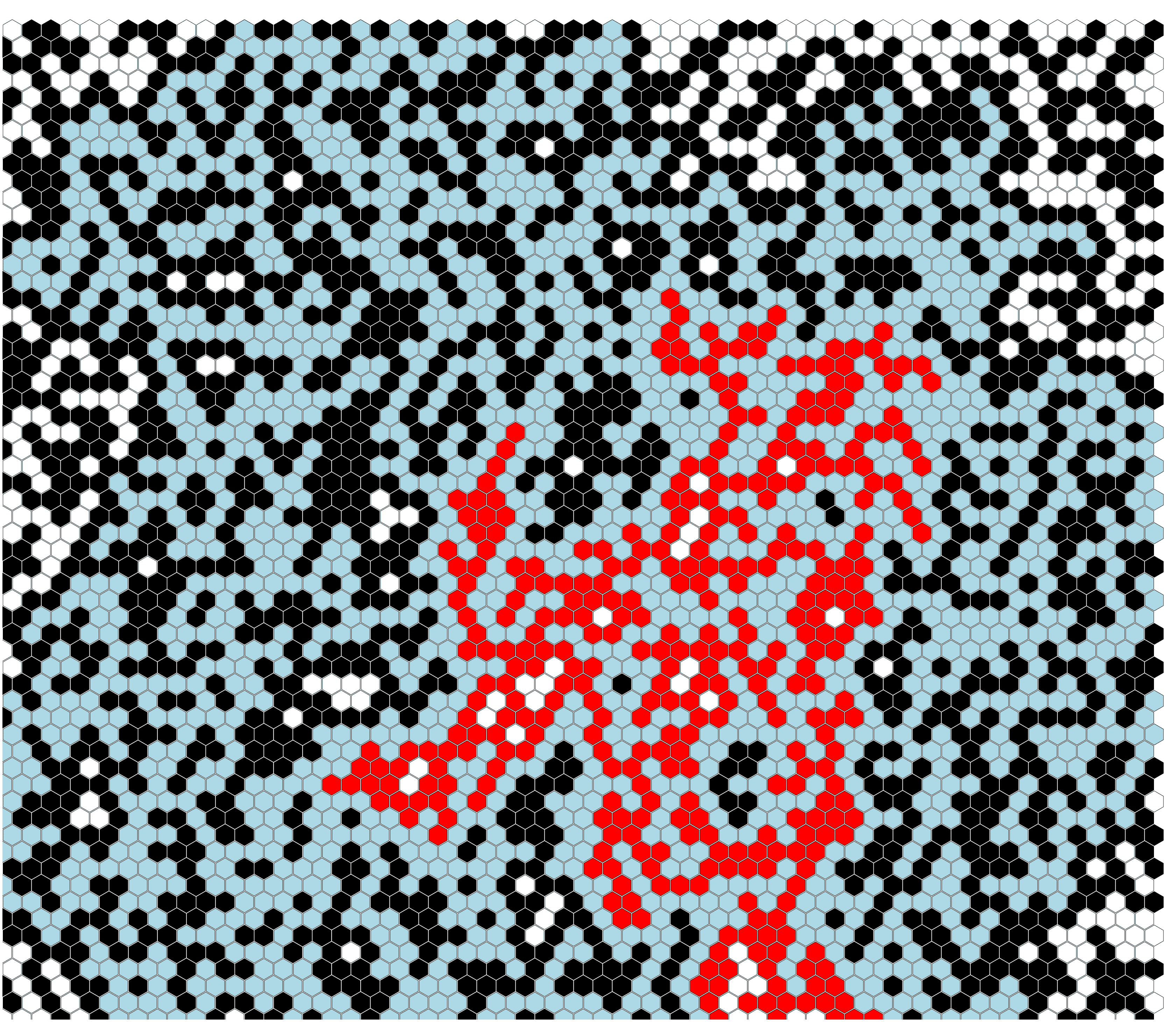}
    \caption{$\ell=0.5$}
    \label{fig:sub4}
  \end{subfigure}

  \caption{Illustration of a configuration $\alpha$ on the hexagonal lattice when the entries of $X$ are independent standard Gaussian random variables. In black, $\alpha>\ell$. In white, $\alpha<\ell$. In red, the biggest component of $\alpha>\ell$. In blue the biggest component of $\alpha<\ell$.}
  \label{fig:4panel}
\end{figure}

In this paper we introduce a new percolation model. Consider a planar discrete field $X : \Z^2 \to \R.$ One may interpret this field as a height function, on each point $u$ of the discrete plane $\Z^2$ there is a mountain of height $X(u).$ Consider a sun located at infinity in the direction $e_1=(1,0).$ We suppose that the sun emits rays that are all parallel and that make a slope $\ell\in \mathbb{R}$ with respect to the horizontal plane. Given this source of light, observers located at the top of each mountain of our landscape may see the sun or be in the shadow of some other mountain located between them and the sun.

Formally, we say that an observer located at $u\in \Z^2$ can see the sun of slope $\ell\in \mathbb{R}$ if and only if
\begin{equation}
    \label{a4:eq:cond_light}
    \forall r\in \mathbb{N}^*,\ X(u)+r\ell \geq X(u+re_1).
\end{equation}

It appears that this problem is related to percolation. Indeed, let $\alpha : \Z^2 \to \overline{\R}$ be the field defined as follows:
\begin{equation}
    \label{a4:eq:def_alpha_general}
    \forall u\in \Z^2,\ \alpha(u):= \sup_{r\in \mathbb{N}^*}\frac{X(u+re_1)-X(u)}{r} \in \overline{\R}
\end{equation}
A key observation is that for any $\ell\in \R$
\begin{equation}
\label{a4:eq:equiv_alpha}
    \alpha(u)\leq \ell \Leftrightarrow \forall r\in \mathbb{N}^*,\ X(u)+r\ell \geq X(u+re_1)
\end{equation}
Therefore, understanding who among the observers see the sun at slope $\ell$ is equivalent to understanding the \textit{excursion set} of $\alpha$, that is
\begin{equation}
    \label{a4:eq:alpha_excursion}
    \{\alpha \leq \ell\} := \{u\in \Z^2\ |\ \alpha(u)\leq \ell\}.
\end{equation}
This is a reformulation of our problem in terms of percolation.
Indeed, in standard percolation theory one takes a domain $S$ (for instance $S=\Z^2$ in our case) and sample a field $f : S \to \R$ according a certain law of probability and then study the properties of the random subsets $\{f\leq \ell\}$ when $\ell$ ranges over $\R$. There exist a variety of ways of sampling $f$ each giving its own percolation model. The most famous model is certainly the Bernoulli percolation model on edges of $\Z^2$. This model was heavily studied (see \cite{SW78}, \cite{Rus78}, \cite{Kes80}, \cite{Gri99} for a far from complete list of references). In this model each edge $e$ of $\Z^2$ is open with some probability $p$ (and $f(e)=1$) and is closed otherwise and $f(e)=0$), all edges being independent from one another. From this model we can obtain several variants, indeed one can play in various directions: replace the graph $\Z^2$ by another graph or lattice, consider site percolation instead of edge percolation, introduce some correlations between edges (there are again various ways to do that). It is also possible to do percolation in the continuum. A first natural example it Poisson percolation, or Poisson Voronoï percolation (see for instance \cite{Ben98}, \cite{BR06}, \cite{Tas16}). Another example in the continuum is the study of the level sets of Gaussian fields (see for instance \cite{BG16},\cite{RV19}, \cite{RV20}, \cite{MV20}, \cite{Sev21}). Although most of these models largely differ from one another, there are still some similarities between some of them.

\paragraph{Symmetries of the domain} The domain $S$ on which the percolation models are defined present natural symmetries. For instance the lattice $\Z^2$ has both horizontal and vertical symmetries and is invariant by rotation of $\pi/2$. The other lattices that are frequently found in the literature are also rich in symmetries (for instance the hexagonal lattice or the union jack lattice). In the continuum the set $\R^d$ also present such symmetries.

\paragraph{Symmetries of the random field} Most of the models used in percolation require the random field $f$ to have a law that preserve certain symmetries of the domain. For instance, in Bernoulli percolation it is clear that if one rotates by $\pi/2$ a random configuration one obtain another configuration which has the same law as the original one.

\paragraph{Stationarity} It is usually the case that the law of the random function $f$ is required to be stationary. This means that if one look at the restriction of the function $f$ in a subset $S_1 \subset S$ or a translated version $S_2$ of $S_1$, then these two restrictions $f_{|S_1}$ and $f_{|S_2}$ must have the same law (of course this implies that the set $S$ admits a good notion of translation, which is the case for $\Z^2$, $\R^d$ etc...)

\paragraph{Independence} Some of the usual models in percolation have an independence property meaning that if two points are far enough from one another then the values of $f$ at these two points are independent.

\paragraph{Conditioning} In the case where we lack the independence property, some models still envoy the fact that conditioning the random field $f$ on some of its values is well behaved. This is the case for Bernoulli percolation (since all entries of $f$ are independent) but also when $f$ is a correlated Gaussian field for instance.

\paragraph{FKG property} This property is more technical. It basically states that two increasing events (events that are favored by an increase of the random field $f$) must be positively correlated. This FKG property is a crucial tool that can be used in various arguments. Note that this property is verified for Bernoulli percolation but also for the percolation of Gaussian fields which have a correlation matrix with positive entries.
\\
\\
In the following we will consider a field $X$ which has the following assumption.
\begin{assumption}
\label{a4:a:a1}
There exists a probability space $(\Omega, \mathcal{A}, \mathbb{P})$ on which is defined a collection $(X(u))_{u\in \Z^2}$ of real random variables which are independent and identically distributed according to $\mu$, a probability distribution on $\R$.
\end{assumption}
In particular, Assumption \ref{a4:a:a1} implies that $X$ verifies all the above nice properties used in percolation. However, most of these properties do not transfer to the field $\alpha$ defined by \eqref{a4:eq:def_alpha_general}. In fact, under Assumption \ref{a4:a:a1}, the field $\alpha$ is stationary but other than this property not much is left. Indeed, among the symmetries of $\Z^2$ that preserve the law of $\alpha$ only the reflection along the horizontal axis remains. The law of the field $\alpha$ will not be invariant by rotation. The field $\alpha$ has long range correlations and does not behave well under conditioning. Moreover, numerical experiments suggest that the field $\alpha$ does not have the FKG property.

In this paper, we study this field $\alpha$ and show that although it lacks many of the nice usual properties used in percolation, it still presents a phase transition when $\ell$ varies. More precisely, under some assumptions of the law $\mu$ of Assumption \ref{a4:a:a1}, we show that when $\ell>0$ is small enough almost surely there exists a unique unbounded connected component of points in the shadow (that is, in $\{\alpha \geq \ell\}$). We also show that when $\ell>0$ is big enough, then almost surely there exists a unique unbounded connected component of points lit by the sun (that is, in $\{\alpha \leq \ell\}$).

We now present our assumptions on the law $\mu$.
\begin{assumption}
\label{a4:a:a2}
The law $\mu$ admits a density $\varphi \in L^1(\R)$ with respect to the Lebesgue measure on $\R,$ that is
$$\forall a\leq b,\ \mu([a,b])=\int_a^b \varphi(x)dx.$$
\end{assumption}
\begin{assumption}
\label{a4:a:a3}
There exist constants $c,C>0$ such that
\begin{equation}
    \forall x\in \R_+,\ \mu(\R\setminus [-x,x])\leq Ce^{-cx}
\end{equation}
\end{assumption}
Our main result concerns site percolation on $\Z^2$ for the field $\alpha$.
\begin{theorem}
\label{a4:thm:principal}
Let $X : \Z^2 \to \R$ be such that Assumption \ref{a4:a:a1} holds and let $\alpha$ be defined as in \eqref{a4:eq:def_alpha_general}. We denote by $\mu$ the law of $X(0)$. In the context of site percolation on $\Z^2$ the following holds.
\begin{itemize}
    \item If $\mu$ satisfies Assumption \ref{a4:a:a2}, there exists $\ell_1>0$ such that, if $\ell<\ell_1$ then,
    \begin{itemize}
        \item almost surely, the set $\{\alpha\leq \ell\}$ does not present an infinite cluster,
        \item almost surely, the set $\{\alpha \geq\ell\}$ presents an infinite cluster.
    \end{itemize}
    \item If $\mu$ satisfies Assumption \ref{a4:a:a3}, there exists $\ell_2<\infty$ such that, if $\ell>\ell_2$ then,
    \begin{itemize}
        \item almost surely, the set $\{\alpha\leq \ell\}$ presents an infinite cluster,
        \item almost surely, the set $\{\alpha \geq\ell\}$ does not present an infinite cluster.
    \end{itemize}
\end{itemize}
\end{theorem}

\paragraph{Comments on Theorem \ref{a4:thm:principal} and open questions}
We comment that Theorem \ref{a4:thm:principal} is first a result in an objective of understanding the phase transition of this new model. Altough Theorem \ref{a4:thm:principal} is stated for a field $X$ defined on $\Z^2$, the proof would work exactly the same if we replaced $\Z^2$ other lattices such as the triangular lattice or the union jack lattice (with appropriate adjustment of the definition \eqref{a4:eq:def_alpha_general} of $\alpha$). For various usual percolation models (such as the Bernoulli percolation model, and Gaussian field percolation for instance), the phase transition is proven to be sharp, that is the map that to the level $\ell$ associates the probability of a big rectangle being crossed by $\{f\leq \ell\}$ converges to a Heaviside function when the size of the rectangle goes to infinity. In particular this implies the existence of a critical parameter $\ell_c\in \R$ such that if $\ell<\ell_c$ then $\{f\leq \ell\}$ has no infinite cluster but $\{f\geq \ell\}$ has one, and the converse is true when $\ell>\ell_c$. It is still an open question to know if, on the hexagonal lattice for instance, we one may take $\ell_1=\ell_2$ in Theorem \ref{a4:thm:principal} (which would prove the existence of such a critical parameter). Such a behaviour is suggested by numerical experiment, see Figure \ref{fig:4panel}. It is also an open question to understand whether or not the phase transition is sharp. Indeed, most strategies to prove sharpness in a phase transition rely either on the FKG inequality (see for instance \cite{Rus82},\cite{BR06},\cite{BDC12}), or on symmetries and invariance presented by the law of the field, notably the invariance by rotation (see \cite{MRVK23}, \cite{Riv21}). Unfortunately, the field $\alpha$ lacks these properties, and the usual arguments cannot be directly applied. Finally, we comment that the model also makes sense when the $X(u)$ are no longer independent but have correlations between them. In future work we plan to address this question. 
\paragraph{Strategy of proof}
We split the proof of Theorem \ref{a4:thm:principal} into two parts. In Section \ref{a4:sec:2} we study the regime of large $\ell$, that is the sun is very high and a large proportion of the mountain peaks should be lit by the sun. In this context a natural idea is to stochastically dominate the set $\{\alpha\geq \ell\}$ by a collection of independent Bernoulli random variable of small parameter. This is a known strategy (see \cite{LSS97} for result on stochastic domination). However, since our field $\alpha$ presents long range correlation, one would first need to define a truncated version of $\alpha$ and study how close it is to the actual field $\alpha$. We prefer to directly use the definition of the field $\alpha$. Our proof relies on understanding what constraints apply to the field $X$ when we ask a few points to be in the shadow even when $\ell$ is large. In Section \ref{a4:sec:3} we study the regime of small $\ell$, that is when the rays of the suns are almost parallel to the horizontal plane. We use the intuition that if the slope is $\ell=0$, then in order for $n$ points aligned in direction of the sun to be all lit by the sun (forgetting what is before them) then those $n$ mountains must have height in a very specific order (that is they must be order by descending height, the highest mountain in the back and the smallest mountain in front). However, we show that  this ordering has very small probability and we then extend this reasoning to $\ell>0$ small enough. In Appendix \ref{a4:sec:appendixA}, we tackle another problem which is to understand how much information about $X$ is stored in the field $\alpha$. We show that under some assumption then, almost surely, the field $X$ may be reconstructed from the field $\alpha$ up to an additive constant.

\paragraph{Notations}
We introduce some notations. We work on the cubic lattices of the form $\mathbb{L}=\Z^2$. Note that this lattice could be replaced by other lattices such as triangular lattices or the union jack lattices, but we choose $\Z^2$ for simplicity).
We say that two points $u=(u_1,u_2)$, $v=(v_1,v_2)$ of $\mathbb{L}$ are \textit{orthogonally adjacent} (and we write $u\sim v$) if $|u_1-v_1|+|u_2-v_2|=1$. We say that these two points are \textit{$\star$-connected} (and we write $u\sim_\star v$) if $\max(|u_1-v_1|,|u_2-v_2|)=1$. Given two vertices $u,v\in \mathbb{L}$, an \textit{orthogonally connected path} (resp. \textit{$\star$-connected path}) between the two vertices $u,v$ is the data of an integer $n\geq 0$ together with a sequence $a_0,a_1,\dots,a_{n-1}$ of vertices of $\mathbb{L}$ such that, $a_0=u$, $a_{n-1}=v$, for all $1\leq i \leq n-1$, $a_{i-1}\sim a_{i}$ (resp. $a_i\sim_\star a_{i+1}$). The interger $n$ is called the \textit{length} of the path. Moreover a path is said to not self-intersect if the map $i\mapsto a_i$ is injective. Given a subset $A\subset \mathbb{L}$ we write $u\overset{A}{\longleftrightarrow} v$ (resp. $u\overset{A}{\underset{\star}{\longleftrightarrow}} v$) if there exists at least one orthogonally connected path (resp $\star$-connected path) between $u$ and $v$ that only uses vertices in $A$. The \textit{orthogonally connected component} (resp. \textit{$\star$-connected component}) in $A$ of a vertex $u$ can be written as $C_A(u):=\{v\in \mathbb{L}\ |\ u\overset{A}{\longleftrightarrow}v\}$ (resp. $C^\star_A(u) :=\{v\in \mathbb{L}\ |\ u\overset{A}{\underset{\star}{\longleftrightarrow}}v\} $). We write $u\overset{A}{\longleftrightarrow}\infty$ if $C_A(u)$ is unbounded.

Given a fixed level $\ell\in \R$ we want to understand the behavior of the sets 
\begin{align}
    \{\alpha \leq \ell\} &:= \{u\in \Z^2\ |\ \alpha(u)\leq \ell\}, \\
    \{\alpha \geq \ell\} &:= \{u\in \Z^2\ |\ \alpha(u)\geq \ell\}.
\end{align}
In particular, for a fixed realization of the field $\alpha$, we see that the map $\ell \mapsto \{\alpha \leq \ell\}$ is non decreasing in $\ell$ for set inclusion, whereas $\ell\mapsto \{\alpha \geq \ell\}$ is non increasing in $\ell$. Denote by $\Theta_{\leq \ell}$ and $\Theta_{\geq \ell}$ the two following quantities.
\begin{align}
    \Theta_{\leq \ell} := \Proba{0\overset{\{\alpha\leq \ell\}}{\longleftrightarrow}\infty}, \\
    \Theta_{\geq \ell} := \Proba{0\overset{\{\alpha \geq \ell\}}{\longleftrightarrow}\infty}.
\end{align}
Then the map $\ell \mapsto \Theta_{\leq \ell}$ is non decreasing in $\ell$ whereas $\ell\mapsto \Theta_{\geq \ell}$ is non increasing in $\ell$.
We define two critical parameters,
\begin{align}
    \ell_c^\leq := \sup\{\ell \in \R \ |\ \Theta_{\leq \ell}=0\}\in \overline{\R},\\
    \ell_c^\geq := \inf\{\ell \in \R\ | \ \Theta_{\geq \ell}=0\}\in \overline{\R},
\end{align}
with the usual conventions $\sup\emptyset=-\infty$ and $\inf\emptyset =+\infty.$
We remark that Theorem \ref{a4:thm:principal} can be reformulated as
\begin{theorem}
\label{a4:thm:discrete}
Under Assumption \ref{a4:a:a1} and \ref{a4:a:a2} then $0<\ell_c^\leq$ and $0<\ell_c^\geq$.
Under Assumptions \ref{a4:a:a1} and \ref{a4:a:a3} then $\ell_c^\leq<\infty$ and $\ell_c^\geq<\infty$.
\end{theorem}

\paragraph{Acknowledgements: } I am grateful to my PhD advisor Damien Gayet who first presented me this model and offered me to study it and also reviewed a first version of this manuscript. I also would like to thank Florent Gimbert who suggested that such models could also be of interest in the understanding of glacier sliding.

\section{The regime of large $\ell$}
\label{a4:sec:2}
The main result of this section is the Proposition \ref{a4:prop:regime_gg_1} which addresses the case of large $\ell$.
\begin{proposition}
\label{a4:prop:regime_gg_1}
Let $X : \Z^2 \to \R$ and $\mu$ be such that Assumption \ref{a4:a:a1} holds. Assume that $\mu$ satisfies Assumption \ref{a4:a:a3}. For any $\varepsilon>0$ there exists $0<\ell_0(\varepsilon)<\infty$ such that if $\ell>\ell_0(\varepsilon)$ the following holds. For any finite subset $A\subset \Z^2$ we have
$$\Proba{\forall u\in A,\ \alpha(u)\geq \ell}\leq \varepsilon^{|A|},$$
where $|A|$ denotes the cardinal of $A$.
\end{proposition}
\begin{remark}
In the special case where $\mu$ is the probability distribution of a standard Gaussian random variable (that is $\mu(dx)=\frac{1}{\sqrt{2\pi}}e^{-x^2/2}dx$) then $\ell$ can be chosen explicitly in terms of $\varepsilon$. Indeed, in this case it can be checked that for all $\ell>0$, for any finite subset $A\subset \Z^2$,
$$\Proba{\forall u\in A,\ \alpha(u)\geq \ell}\leq  \left(\frac{q(2-q)}{1-q}\right)^{|A|},$$
where $q=e^{-\ell^2/4}$.
\end{remark}
We argue that by the classical Peierls argument, see also \cite{MS83_1} and \cite{MS83_2}, Proposition \ref{a4:prop:regime_gg_1} implies $\ell_c^\leq<\infty$ and $\ell_c^\geq<\infty$.
\begin{corollary}
\label{a4:cor:discrete_indep}
Let $X : \Z^2 \to \R$ and $\mu$ be such that Assumption \ref{a4:a:a1} holds. Assume that $\mu$ satisfies Assumption \ref{a4:a:a3}.
Then, $\ell_c^\leq<\infty$ and $\ell_c^\geq<\infty$.
\end{corollary}
The proof is completely classical and we provide it for the convenience of the reader.
\begin{proof}[Proof of Corollary \ref{a4:cor:discrete_indep}]
     Let $\varepsilon>0$ be small parameter and $\ell_0$ associated to it by Proposition \ref{a4:prop:regime_gg_1}. Let $\ell>\ell_0$ and $n\in \N^*$.
     
     Let us first show that Proposition \ref{a4:prop:regime_gg_1} implies that $\ell_c^\geq<\infty$. In $\Z^2$ there are at most $4^n$ orthogonally connected paths which do not self-intersect that start from $(0,0)$ and of length $n$. For the event $\{0\overset{\{\alpha\geq \ell\}}{\longleftrightarrow}\infty\}$ to occur, there must be at least one of these paths of length $n$ such that $\alpha(u)\geq \ell$ for all $u$ in this path.
    However, for a fixed path $\gamma$ of length $n$, the probability that $\alpha(u)\geq \ell$ for all $u$ in the path $\gamma$ is bounded from above by $\varepsilon^n$ (by Proposition \ref{a4:prop:regime_gg_1}). We can do an union bound on the at most $4^n$ such paths to get
    \begin{equation}
    \label{a4:eq:p1}
        \forall n\geq 1,\ \Proba{0\overset{\{\alpha\geq \ell\}}{\longleftrightarrow}\infty}\leq \left(4\varepsilon\right)^n
    \end{equation}
    By choosing $\varepsilon <\frac{1}{4}$ and letting $n$ go to infinity in \eqref{a4:eq:p1} we get
    \begin{equation}
            \forall \ell\geq \ell_0,\ \Proba{0\overset{\{\alpha\geq \ell\}}{\longleftrightarrow}\infty}=0.
    \end{equation}
    This shows $\ell_c^\geq<\infty.$
    
    Let us now show that Proposition \ref{a4:prop:regime_gg_1} implies that $\ell_c^\leq<\infty$. In fact for the event $\{0\overset{\{\alpha < \ell\}}{\longleftrightarrow}\infty\}$ to not occur, there must exist some $n\geq 1$ and some non self-intersecting circuit $\gamma$ of length $n$ separating $0$ from infinity such that for all $u$ in $\gamma$ we have $\alpha(u)\geq \ell$. Note that for such a circuit to disconnect $0$ from $\infty$, then this circuit does not necessarily requires to be orthogonally connected, but requires to be $\star-$connected. Therefore, the number of such circuits of length $n$ is no more than $n\times8^n$ (in fact one may choose the first point of the circuit in $\{0,\dots, n\}\times \{0\}$ and each vertex has at most $8$ neighbors for $\sim_\star$).
    We can thus obtain the following crude upper bound,
    \begin{equation}
        \label{a4:eq:p2}
        1-\Proba{\{0\overset{\{\alpha < \ell\}}{\longleftrightarrow}\infty\}} \leq \sum_{n\geq 1}n8^n\varepsilon^n.
    \end{equation}
    A straightforward computation shows that the right hand side in \eqref{a4:eq:p2} is finite and strictly less than $1$ if we choose $\varepsilon>0$ small enough. Therefore, we may find $\ell_0\in ]0,\infty[$ (associated to this $\varepsilon$) such that,
    $$\forall \ell>\ell_0,\ \Proba{0\overset{\{\alpha < \ell\}}{\longleftrightarrow}\infty}>0.$$
    This implies $\ell_c^\leq<\infty.$
\end{proof}
The rest of this section is devoted to the proof of Proposition \ref{a4:prop:regime_gg_1}. We will restrict ourselves to the case of a subset $A\subset \Z \times \{0\}$. In fact given the definition of $\alpha$ in \eqref{a4:eq:def_alpha_general} we can see that the restriction of $\alpha$ to the lines $\Z\times \{n\}$, where $n$ ranges over $\Z$, are independent and identically distributed.

To ease notations, let us write $Y_i=X_{(i,0)}$ for $i\in \Z$. Then $(Y_i)_{i\in \Z}$ is a collection of independent standard Gaussian random variables. Let us fix $\ell>0$.
We define a relation $\prec_\ell$ between points of $\Z$ as follows:
\begin{definition}
    For $i,j\in \Z$ we write $i\prec_\ell j$ if $i<j$ and $Y_j>Y_i+(j-i)\ell$.
\end{definition}
\begin{remark}
This relation is transitive, if $i\prec_\ell j$ and $j\prec_\ell k$ then $i\prec_\ell k$. In terms of shadow, if the rays of the sun make a slope $\ell$ with the horizontal, then the mountain $(j,Y_j)$ casts shadow on the mountain $(i,Y_i)$ if and only if $i\prec_\ell j$. Moreover, observe than we can rewrite $\alpha(i)>\ell \Leftrightarrow \exists j>i, i\prec_\ell j.$
\end{remark}
In the following we will make repeated use of the following consequence of Assumptions \ref{a4:a:a1} and \ref{a4:a:a3}.
\begin{claim}
    \label{a4:claim:F}
    Assume that $X$ satisfies Assumptions \ref{a4:a:a1} and \ref{a4:a:a3}. Then there exists $c'>0$ and $x_0>0$ such that
    \begin{equation}
        \forall x\geq x_0,\ \forall u\neq v\in \Z^2,\ \Proba{|X(u)-X(v)|\geq x}\leq e^{-c'x}.
    \end{equation}
\end{claim}
\begin{proof}
    By Assumption \ref{a4:a:a1}, $X(u)$ and $X(v)$ are independent and distributed according to $\mu.$
    We have by an union bound
    \begin{align*}
        \Proba{|X(u)-X(v)|\geq x} &\leq \Proba{|X(u)|\geq \frac{x}{2}}+\Proba{|X(v)|\geq \frac{x}{2}} \\
        &\leq 2\mu(\R\setminus [-x/2,x/2]).
    \end{align*}
    By Assumption \ref{a4:a:a3}, we may find $C>0$ $c>0$ such that $\mu(\R \setminus [-x,x])\leq Ce^{-cx},$ therefore if we set $c'=c/4>0$ we may find $x_0>0$ such that for $x\geq x_0$,
    \begin{align*}
        \Proba{|X(u)-X(v)|\geq x} \leq 2Ce^{-cx/2} \leq e^{-c'x}.
    \end{align*}
\end{proof}
In order to prove Proposition \ref{a4:prop:regime_gg_1}, we make a few definitions.
\begin{definition}
    Let $\ell>0$ and $i\in \Z$. We define the random set $M_\ell(i)$ as
    \begin{equation*}
        M_\ell(i):= \{Y_j - (j-i)\ell\ |\ j>i\}.
    \end{equation*}
    We define the event $\mathcal{T}_\ell(i)$ as
    \begin{equation*}
        \mathcal{T}_\ell(i):= \{\text{the set }M_\ell(i) \text{ admits a maximum}\}.
    \end{equation*}
    On the event $\mathcal{T}_\ell(i)$ we define a random integer $T_\ell(i)$ as
    \begin{equation*}
        T_\ell(i) := \min\{j>i\ |\ Y_j-(j-i)\ell = \max M_\ell(i)\}.
    \end{equation*}
\end{definition}
\begin{remark}
    In the definition of $T_\ell(i)$ we decided to select the minimum index $j$ that satisfies the maximum condition, this choice if just for convenience. Actually, if we also assume Assumption \ref{a4:a:a2}, then it would not be hard to check that for fixed $\ell$ then, almost surely, for all $i\in \Z$, there cannot be two indexes that realize the maximum of $M_\ell(i)$. But for our purpose, this arbitrary choice is not important. What matters is that $T_\ell(i)$ can be understood as the index of a point that generates the "highest" shadow on $i$ (at level $\ell$).
\end{remark}
\begin{claim}
Assume that Assumptions \ref{a4:a:a1} and \ref{a4:a:a3} hold.
    For any $\ell>0$ and $i\in \mathbb{Z}$, the event $\mathcal{T}_\ell(i)$ has probability $1$.
\end{claim}
\begin{proof}
    Let $i\in \Z$. We have
    $$\Proba{(\mathcal{T}_\ell(i))^c} \leq \Proba{(Y_j -(j-i)\ell > Y_{i+1}-\ell \text{ for infinitely many }j>i}.$$ This last probability is zero by the Borel-Cantelli Lemma and using Fact \ref{a4:claim:F}.
\end{proof}
For fixed $\ell>0$, we now work on $\Omega_\ell = \bigcap_{i\in \Z}\mathcal{T}_\ell(i)$ which has full probability.
\begin{claim}
    \label{a4:claim:1}
    Let $i\in \Z$. We have $\alpha(T_\ell(i))\leq \ell$. Moreover, if $\alpha(i)>\ell$, then $i\prec_\ell T_\ell(i)$. 
\end{claim}
\begin{remark}
    If we keep the intuition of $T_\ell(i)$ as the mountain that creates the highest shadow on $i$. Then the first part of Fact \ref{a4:claim:1} can be interpreted by saying that the mountain creating the highest shadow on us cannot itself be in the shadow of someone else.
\end{remark}
\begin{proof}[Proof of Fact \ref{a4:claim:1}]
First, let us show that $\alpha(T_\ell(i))\leq \ell$. By contradiction, if that is not the case we would find some integer $j>T_\ell(i)$ such that $T_\ell(i)\prec_\ell j$, then
     $$Y_{T_\ell(i)}< Y_j - (j-T_\ell(i))\ell.$$
     This would imply
     $$Y_{T_\ell(i)}-(T_\ell(i)-i)\ell < Y_j -(j-i)\ell.$$
     Such an inequality contradicts the definition of $T_\ell(i)$.
     
     Now, let us show that if $\alpha(i)>\ell$ then $i\prec_\ell T_\ell(i)$. Since $\alpha(i)>\ell$ there exists $j_0>i$ such that $i\prec_\ell j_0$. This rewrites as $Y_{j_0}-(j_0-i)\ell > Y_i$. By definition of $T_\ell(i)$, we have $$Y_{T_\ell(i)}-(T_\ell(i)-i)\ell \geq Y_{j_0}-(j_0-i)\ell > Y_i.$$ Thus, $$i\prec_\ell T_\ell(i).$$
\end{proof}

\begin{claim}
    \label{a4:claim:2}
    Let $i<j$ be two points of $\Z$ such that $\alpha(i)>\ell$ and $\alpha(j)>\ell$. If $i\not\prec_\ell T_\ell(j)$ then $i<T_\ell(i)<j$.
\end{claim}
\begin{remark}
We may interpret Fact \ref{a4:claim:2} the following way: If two points $i<j$ are in the shadow but if $i$ is not in the shadow of the principal responsible for the shadow on $j$, then the principal responsible for the shadow on $i$ lies between $i$ and $j$.
\end{remark}
\begin{proof}
    First, we always have $T_\ell(i)>i$ (by definition).\\
    By contradiction, if $T_\ell(i)>j$, then by definition of $T_\ell(j)$ we have
    $$Y_{T_\ell(j)}-(T_\ell(j)-j)\ell \geq Y_{T_\ell(i)}-(T_\ell(i)-j)\ell.$$
    By substracting $(j-i)\ell$ from both sides we find $$Y_{T_\ell(j)}-(T_\ell(j)-i)\ell \geq Y_{T_\ell(i)} - (T_\ell(i)-i)\ell.$$
    However, since $i\prec_\ell T_\ell(i)$ (by Fact \ref{a4:claim:1}) and since $i\not\prec_\ell T_\ell(j)$ (par assumption), we have
    $$Y_i \geq Y_{T_\ell(j)}-(T_\ell(j)-i)\ell \geq Y_{T_\ell(i)} - (T_\ell(i)-i)\ell > Y_i,$$
    which is a contradiction.\\
    It only remains to argue that $T_\ell(i)\neq j$. In fact, by assumption we have $\alpha(j)>\ell$ but by Fact \ref{a4:claim:1} we have $\alpha(T_\ell(i))\leq \ell$ which concludes the proof.
\end{proof}

We now provide the proof of Proposition \ref{a4:prop:regime_gg_1}.
\begin{proof}[Proof of Proposition \ref{a4:prop:regime_gg_1}]
For convenience, we see $\Z$ as a subset of $\Z^2$ by the natural inclusion $n\mapsto (n,0).$
We first do the proof for a finite subset of the form $A\subset  \Z$ (that is, we must understand $A$ as a subset of $\Z\times \{0\}$).
    If $A=\emptyset$, there is nothing to prove. In the following we suppose $|A|\geq 1$, and we note $A=\{a_1,a_2,\dots,a_n\}$ with $a_1<a_2<\dots<a_n$ (thus $n=|A|$). First, on the event
    \begin{equation}
        \{\forall i\in A\ |\ \alpha(i)>\ell\}.
    \end{equation}
    we will define and integer $r\geq 0$ and a decreasing sequence of integers $(k_i)_{0\leq i \leq r}$ such that $n=k_0>k_1>\dots>k_r\geq 1$.
    We build this sequence by induction.
    We first set $k_0=n$. Let us assume that we are a step $r'\geq 0$, $(k_m)_{0\leq m\leq r'}$ was built and the construction is not over. If for every $l\in \llbracket 1, k_{r'}-1\rrbracket$ we have $a_l \prec_\ell T_\ell(a_{k_r'})$ then we end the construction (and $r=r'$). Otherwise, let $k_{r'+1}$ be the biggest integer strictly smaller than $k_{r'}$ such that $a_{k_{r'+1}}\not\prec_\ell T_\ell(a_{k_{r'}})$. We add $k_{r'+1}$ to our sequence, and we then reproduce this procedure until it ends.\\\\
    since $A$ is finite, the procedure eventually ends. Moreover, Facts \ref{a4:claim:1} and \ref{a4:claim:2} imply that on the event $\{\forall i \in A\ |\ \alpha(i)>\ell\}$ then the sequence $(k_i)_{0\leq i \leq r}$ has the following properties:
    \begin{enumerate}
        \item $\forall l\in \llbracket 1, k_r\rrbracket,\ a_l \prec_\ell T_\ell(a_{k_r})$
        \item $\forall m\in \llbracket 1,r\rrbracket,\ \forall l\in \llbracket k_m+1, k_{m-1}\rrbracket,\ a_l \prec_\ell T_\ell(a_{k_{m-1}}).$
        \item $\forall m\in \llbracket 1, r\rrbracket,\  a_{k_m} \not\prec_\ell T_\ell(a_{k_{m-1}})$
        \item $a_{k_r}<T_\ell(a_{k_r})<a_{k_{r-1}}<T_\ell(a_{k_{r-1}}) <\dots < a_{k_1}<T_\ell(a_{k_1}) < a_{k_0}=a_n < T_\ell(a_{k_0})=T_\ell(a_n)$
    \end{enumerate}
    Note that in the case $r=0$, then $k_r=k_0=n$ and the second and third property are empty. The first property comes from an application of Fact \ref{a4:claim:1} for $l=k_r$ and of our ending condition for the algorithmic procedure. The second property also comes from Fact \ref{a4:claim:1} for $l={k_{m-1}}$ together with the definition of the induction. The third property is a consequence of the definition of the sequence $(k_i)_{0\leq i \leq r}$.  The last property follows from Fact \ref{a4:claim:2} and the third property.

    Now that we have built this sequence $(k_i)_{0\leq i \leq r}$ we introduce a few more notations.
    For $0\leq m\leq r$, we denote by $j_m$ the index
    $$j_m:= T_\ell(a_{k_m}).$$
    We now define the set $J$ as
    \begin{equation}
            J=\{j_r, \dots, j_0\}.
    \end{equation}
    One may think of $J$ as a minimal set of points which is enough to cast shade to all points in $A$.
    We observe that $J\cap A=\emptyset$. Indeed, the elements $a\in A$ satisfy $\alpha(a)>\ell$ whereas, by Fact \ref{a4:claim:1} the elements of $j\in J$ satisfy $\alpha(j)\leq \ell$.
    Finally we observe that $|J|$ can vary between $|J|=1$ (when only point is enough to cast shade onto all points in $A$) and $|J|=n=|A|$ (when all $T_\ell(a_i)$ are different). Therefore, we write
    \begin{equation}
        \mathbb{P}(\{\forall i\in A, \alpha(i)>\ell\}) = \sum_{r=0}^{n-1} \mathbb{P}(\{\forall i\in A, \alpha(i)>\ell\}\cap \{|J|=r+1\}).
    \end{equation}
    For $0\leq r\leq n-1$, we denote by $E_r$ the event
    $$E_r:= \{\forall i\in A,\ \alpha(i)>\ell\}\cap \{|J|=r+1\}.$$
    Consider $c',x_0$ given by Fact \ref{a4:claim:F}.
    We work with $\ell> x_0$, we denote by $\rho := e^{-c'\ell}\in ]0,1[$ and we observe that if $i<j$ then $$\Proba{i\prec_\ell j}\leq \rho^{j-i}.$$
    For $r=0$ we find
    \begin{align*}
        \Proba{E_0}& \leq \sum_{j_0=a_n+1}^\infty\mathbb{P}(a_1\prec_\ell j_0) \\
        &\leq \sum_{j_0=a_n+1}^\infty \rho^{j_0-a_1} \\
        &\leq \frac{\rho^{a_n+1-a_1}}{1-\rho} \\
        &\leq \frac{\rho^n}{1-\rho}
    \end{align*}
    In the first line, we use the fact that $J=\{j_0\}$ for some $j_0>a_n$ and we do an union bound on the possible values for $j_0$ (we also only keep the constraint $a_1\prec_\ell j_0$).
    From the first, to second line, we apply Fact \ref{a4:claim:F} and we use the independence of $Y_{a_1}$ and $Y_{j_0}$. In the last two lines, we used the fact that $\rho=e^{-c'\ell}$ is between $0$ and $1$ and the fact that $a_n-a_1\geq n-1.$
    
    Here is a similar (although more technical) computation for $r\geq 1$.
    \begin{align*}
        \mathbb{P}(E_r) &\leq 
\sum_{j_0=a_n+1}^\infty
\sum_{(k_i)_{0\leq i \leq r}}
\sum_{\substack{j_m, 1\leq m\leq r\\
  a_{k_m}<j_m <a_{k_{m-1}}\\ j_m \not\in A}}
\mathbb{P}\left(
  \begin{array}{c}
    a_1 \prec_\ell j_r \text{ and } \\
    a_{k_m+1}\prec_\ell j_{m-1}\ \forall m\in \llbracket 1,r\rrbracket
  \end{array}
\right).
\end{align*}
        where the second sum is over all sequences $(k_i)_{0\leq i \leq r}$ such that $1\leq k_r<k_{r-1}<\dots < k_1< k_0=n$ (note that there are $\binom{n-1}{r}$ terms in this sum), and the third sum is over all $(j_1,\dots, j_m)$ such that $a_{k_m}<j_m<a_{k_{m-1}}$ and $j_m\not\in A$ for all $m$. (note that we can sum over $j_m\not\in A$ because $J\cap A=\emptyset$). The events $a_{k_{m}+1}\prec_\ell j_{m-1}$ and $a_1\prec_\ell j_r$ are all independent (this stems from the fact that $J\cap A=\emptyset$ and from the independence of the $(Y_i)_{i\in \Z}$). This allows us to compute an upper bound for $\Proba{E_r}$.
        \begin{align*}
        \mathbb{P}(E_r)\leq& \sum_{j_0=a_n+1}^\infty\sum_{(k_i)_{0\leq i \leq r}}\sum_{\substack{j_m, 1\leq m\leq r\\
        a_{k_m}<j_m <a_{k_{m-1}}\\ j_m\not\in A}}\rho^{j_r-a_1}\prod_{m=1}^r \rho^{j_{m-1}-a_{k_m+1}}\\
        \leq &\sum_{j_0=a_n+1}^\infty \rho^{j_0-a_n}\sum_{(k_i)_{0\leq i \leq r}}\rho^{a_{n}-a_{k_1+1}}\left(\sum_{j_r=a_{k_r}+1}^\infty \rho^{j_r-a_1}\right) \\
        & \times \prod_{m=2}^r \left(\sum_{j_{m-1}=a_{k_{m-1}}+1}^\infty \rho^{j_{m-1}-a_{k_m+1}}\right) \\
        =& \sum_{j_0=a_n+1}^\infty \rho^{j_0-a_n}\sum_{(k_i)_{0\leq i \leq r}}\rho^{a_{n}-a_{k_1+1}}\frac{\rho^{a_{k_r}+1-a_1}}{1-\rho}\prod_{m=2}^r \frac{\rho^{a_{k_{m-1}}+1-a_{k_m+1}}}{1-\rho} \\
        \leq & \sum_{j_0=a_n+1}^\infty \rho^{j_0-a_n}\sum_{(k_i)_{0\leq i \leq r}}\rho^{n-k_1-1}\frac{\rho^{k_r+1-1}}{1-\rho}\prod_{m=2}^r \frac{\rho^{k_{m-1}+1-k_m-1}}{1-\rho}\\
       = &\sum_{j_0=a_n+1}^\infty \rho^{j_0-a_n}\sum_{(k_i)_{0\leq i \leq r}} \frac{\rho^{n-1}}{(1-\rho)^r} \\
       = &\binom{n-1}{r}\frac{\rho^n}{(1-\rho)^{r+1}}.
        \end{align*}
    Summing all these inequalities for $r$ going from $0$ to $n-1$ we get
    \begin{align*}
        \mathbb{P}(\forall i\in A, \alpha(i)>\ell) &\leq \sum_{r=0}^{n-1}\mathbb{P}(E_r) \\
        &\leq \sum_{r=0}^{n-1}\binom{n-1}{r}\frac{\rho^n}{(1-\rho)^{r+1}}\\
        &= \frac{\rho^n}{1-\rho}\left(\frac{2-\rho}{1-\rho}\right)^{n-1}\\
        &\leq \left(\frac{\rho(2-\rho)}{1-\rho}\right)^n.
    \end{align*}
Now, let $\varepsilon>0$. We choose $\ell\geq x_0$ big enough so that $\rho=e^{-c'\ell}$ is small enough so that $\left(\frac{\rho(2-\rho)}{1-\rho}\right) <\varepsilon.$
This concludes the proof of Proposition \ref{a4:prop:regime_gg_1} for $A\subset \Z\times \{0\}$. To deduce the result for $A\subset \Z^2$, observe that any finite set $A\subset \Z^2$ can be written as $A=\sqcup_{i\in \Z} A_i$ where $A_i$ is a subset of $\Z\times \{i\}$ (in fact $A_i = A\cap (\Z\times \{i\})$. Since the collections $(\alpha(u))_{u\in \Z\times \{i\}}$ are mutually independent and identically distributed when $i$ ranges over $\Z$ we get
\begin{align*}
    \Proba{\forall u\in A,\ \alpha(u)>\ell} &=\Proba{\forall i\in \Z,\forall u\in A_i,\ \alpha(u)>\ell}\\
    &=\prod_{i\in \Z}\Proba{\forall u \in A_i,\ \alpha(u)>\ell} \\
    &\leq \prod_{i\in \Z} \varepsilon^{|A_i|} \\ &= \varepsilon^{\sum_{i\in \Z}|A_i|}\\&=\varepsilon^{|A|}.
\end{align*}
This concludes the proof of Proposition \ref{a4:prop:regime_gg_1}.
\end{proof}

\section{The regime of small $\ell$}
\label{a4:sec:3}
We now focus our attention on the behavior of our system when $\ell$ is a small positive number. The main result of this section is Proposition \ref{a4:prop:regime_ll_1} which is the analogue of Proposition \ref{a4:prop:regime_gg_1} for the regime of small positive $\ell$.
\begin{proposition}
\label{a4:prop:regime_ll_1}
Let $X : \Z^2 \to \R$ and $\mu$ be such that Assumption \ref{a4:a:a1} holds. Assume that $\mu$ satisfies Assumption \ref{a4:a:a2}.
Let $\varepsilon>0$. There exists $\ell_0(\varepsilon)>0$ such that, for any $\ell\in [0,\ell_0(\varepsilon)]$, for any finite subset $A\subset \Z^2$ we have
\begin{equation}
    \Proba{\forall u\in A,\ \alpha(u)\leq \ell}\leq \varepsilon^{|A|}.
\end{equation}
\end{proposition}
Here is an immediate corollary of Proposition \ref{a4:prop:regime_ll_1}.
\begin{corollary}
\label{a4:cor:2}
Let $X : \Z^2 \to \R$ and $\mu$ be such that Assumption \ref{a4:a:a1} holds. Assume that $\mu$ satisfies Assumption \ref{a4:a:a2}.
Then, $\ell_c^\leq>0$ and $\ell_c^\geq>0.$
\end{corollary}
\begin{proof}
    The proof is completely similar to the proof of Corollary \ref{a4:cor:discrete_indep} (using Proposition \ref{a4:prop:regime_ll_1} in place of Proposition \ref{a4:prop:regime_gg_1}).
\end{proof}
We comment that we now can prove Theorem \ref{a4:thm:principal}.
\begin{proof}[Proof of Theorem \ref{a4:thm:principal}]
    Theorem \ref{a4:thm:principal} directly follows from Corollaries \ref{a4:cor:discrete_indep} and \ref{a4:cor:2}.
\end{proof}
The rest of this section is devoted to the proof of Proposition \ref{a4:prop:regime_ll_1}. As previously we start working on the one dimensional system, we will write $Y_i=X_{i,0}$ for $i\in \Z$. We introduce the following notation.
\begin{definition} Let $h \in \R$. Let $x,y\in \R$. We write 
    $x>_h y$ for $x+h>y$. We will also write $x_1>_h x_2 >_h x_3 >_h \dots$ to mean $x_1+h>x_2$, $x_2+h>x_3$ etc.
\end{definition}
The next lemma is the key ingredient to the proof of Proposition \ref{a4:prop:regime_ll_1}.
\begin{lemma}
    \label{a4:lemma:2_1}
    Let $X : \Z^2 \to \R$ and $\mu$ be such that Assumptions \ref{a4:a:a1} and \ref{a4:a:a2} hold.
    Let $\varepsilon>0$. There exist $n_0\in \N^*$ and $h_0>0$ (depending on $\varepsilon$) such that
    \begin{equation*}
        \forall h\leq h_0,\ \mathbb{P}(Y_1>_h Y_2>\dots >_h Y_{n_0})\leq \varepsilon^{2n_0}.
    \end{equation*}
\end{lemma}
\begin{proof}
    The key observation is that for any $n_0\geq 1$ we have
    \begin{equation}
        \label{a4:eq:key_factorial}
        \mathbb{P}(Y_1>Y_2>\dots >Y_{n_0})=\frac{1}{n_0!}.
    \end{equation}
    Indeed, \eqref{a4:eq:key_factorial} is due to the fact that the probability of having a specific order for $n_0$ independent random variable with density does not depend on this specific order (we also implicitly used the fact that almost surely, two of the $Y_i$ cannot be the same).

    Let $n_0$ be such that
    $$\frac{1}{n_0!}\leq \frac{\varepsilon^{2n_0}}{2}.$$
    For this specific $n_0$ which is fixed (and which depends on $\varepsilon$) we observe that
    \begin{equation}
        \label{a4:eq:key_factorial_2}
        \mathbb{P}(Y_1>_h Y_2>\dots >_h Y_{n_0}) \xrightarrow[h\to 0]{}  \mathbb{P}(Y_1> Y_2>\dots > Y_{n_0}) = \frac{1}{n_0!}
    \end{equation}
    In fact, \eqref{a4:eq:key_factorial_2} comes from the continuity in $h$ of the left hand side (which can be seen by writing this probability as an integral and using Assumption \ref{a4:a:a2}).
    
    Therefore, we may find $h_0>0$ depending on $\varepsilon$ such that
    \begin{equation}
        \forall h\leq h_0, \ \mathbb{P}(Y_1>_h Y_2>\dots >_h Y_{n_0}) < \varepsilon^{2n_0}.
    \end{equation}
    This is precisely the conclusion.
\end{proof}
\begin{definition}
    Let $R\geq 1$. We introduce the field $\alpha_R$ as a truncated version of the field $\alpha$, that is
    $$\forall u \in \Z^2,\ \alpha_R(u) := \max_{r\in \llbracket 1 , R\rrbracket} \frac{X_{u+re_1}-X_u}{r}.$$
\end{definition}
\begin{remark}
Note that the difference between $\alpha_R$ and $\alpha$ is that instead of taking the supremum over $r\in \mathbb{N}^*$ for $\alpha$ we restrict ourselves to $r\in \llbracket 1, R\rrbracket$ for $\alpha_R$. As such, we may understand $\alpha_R$ as a variant of $\alpha$ where we only consider the shadows coming from the $R$ first points ahead of us. A consequence of the definition is that, if $\alpha(u)\leq \ell$ then $\alpha_R(u)\leq \ell$ (for any $u,\ell,R$).
\end{remark}
\begin{lemma}
    \label{a4:lemma:2_2}
    Let $\eta>0$. There exist $R_0\geq 1$ and $\ell_0>0$ such that
    \begin{equation}
        \forall u\in \Z^2,\ \forall \ell\in [0,\ell_0],\ \Proba{\alpha_{R_0}(u)\leq \ell}\leq \eta
    \end{equation}
\end{lemma}
\begin{proof}
    By stationarity, it is enough to prove the result for $u_0=(0,0)$.
    First, note that for any given $R_0\geq 1$ we have
    \begin{equation}
        \label{a4:eq:factorial_3}
         \Proba{\alpha_{R_0}(u_0)\leq 0}=\frac{1}{R_0+1}.
    \end{equation}
    Indeed, for $\alpha_{R_0}(u_0)$ to be negative, one need to have $X_{u_0}$ to be greater than any of the $X_{u_0+re_1}$ for $1\leq r\leq R$. So the above probability, is the probability that $X_{u_0}=\max_{0\leq r\leq R}{X_{u_0+re_1}}.$ Since the $(X_v)_{v\in \Z^2}$ are independent and with density, this probability is precisely $\frac{1}{R_0+1}$ (there is the same probability for any point in this finite sequence to be the maximum of the sequence).
    
    Due to \eqref{a4:eq:factorial_3}, we may find $R_0\geq 1$ depending on $\eta$ such that
    \begin{equation}
        \label{a4:eq:factorial_4}
        \Proba{\alpha_{R_0}(u_0)\leq 0}\leq \frac{\eta}{2}.
    \end{equation}
    Now, we argue that the map $\ell\mapsto \Proba{\alpha_{R_0}(u_0)\leq \ell}$ is continuous. In fact one may write
    \begin{align*}
        \Proba{\alpha_{R_0}(u_0)\leq \ell}&=\Proba{\forall r\in \llbracket 1, R_0\rrbracket,\  X_{u+re_1}< X_{u_0}+r\ell} \\
        &=\int_{\R}\varphi(x)\prod_{r=1}^{R_0}F(x+r\ell) dx,
    \end{align*}
    where $\varphi\in L^1(\R)$ is the density of $X_{u_0}$ and $F$ is the cumulative distribution function associated to this density ($F$ is continuous, positive and bounded by $1$). The dominated convergence theorem ensures continuity in $\ell$.
    This continuity together with \eqref{a4:eq:factorial_4} conclude the proof.
\end{proof}

\begin{definition}
    Let $A\subset \Z$ be a finite subset, $A=\{a_1,a_2,\dots,a_n\}$ with $a_1<a_2<\dots<a_n$.
    Let $R_0\geq 1$. We say that $A$ is \textit{$R_0$-connected} if
    \begin{equation}
        \forall i\in \llbracket 1,n-1\rrbracket,\ |a_{i+1}-a_i|\leq R_0.
    \end{equation}
\end{definition}
\begin{lemma}
    \label{a4:lemma:decomp_R0_connected}
    Let $A\subset \Z$ be a finite subset. Let $R_0\geq 1$. There exist $k\in \N$ and subsets $A_1,\dots,A_k$ of $A$ such that
    \begin{enumerate}
        \item $A_i$ if $R_0$-connected for any $i\in \llbracket 1, k\rrbracket.$
        \item $A=\sqcup_{i=1}^k A_i$
        \item $\forall 1\leq i\leq k-1, \max A_i + R_0 < \min A_{i+1}.$
    \end{enumerate}
    Such subsets of $A$ form what we call a \emph{$R_0-$decomposition} of $A$.
\end{lemma}
\begin{proof}
We do the proof by induction on the cardinal of $A$.
If $A=\emptyset$, we may take $k=0$ and there is nothing to check. Assume the property holds for sets of cardinal $n$. Let us consider $A$ a set of cardinal $n+1$ which we may write as $A=\{a_1,\dots,a_{n+1}\}$ with $a_1<a_2<\dots<a_{n+1}.$
Then consider $\tilde{A}=A\setminus \{a_{n+1}\}$. We may find a $R_0$-decomposition of $\tilde{A}$ which we denote by $(\tilde{A}_i)_{1\leq i \leq k}.$ If $a_{n+1}-a_n \leq R_0$ we may add $a_{n+1}$ to $\tilde{A}_k$, and leave the other $\tilde{A}_i$ unchanged, which gives an $R_0$-decomposition of $A$. Otherwise we may define $\tilde{A}_{k+1}:=\{a_{n+1}\}$. Then $(\tilde{A}_i)_{1\leq i \leq k+1}$ is a $R_0$-decomposition of $A$.
\end{proof}
We now prove Proposition \ref{a4:prop:regime_ll_1}.
\begin{proof}[Proof of Proposition \ref{a4:prop:regime_ll_1}]
    Let $\varepsilon\in ]0,1[$. Applying Lemma \ref{a4:lemma:2_1} we let $h_0>0$ and $n_0\geq 1$ (depending on $\varepsilon$) be such that :
    \begin{equation}
        \label{a4:eq:f1}
        \forall 0\leq h\leq h_0,\ \mathbb{P}(Y_1>_h Y_2>\dots >_h Y_{n_0})\leq \varepsilon^{2n_0}.
    \end{equation}
    We apply Lemma \ref{a4:lemma:2_2} for $\eta = \varepsilon^{2n_0}$. We may therefore define $\ell_0>0$ and $R_0\geq 1$ (depending on $\varepsilon$) be such that
    \begin{equation}
        \label{a4:eq:f2}
        \forall u\in \Z^2,\  \forall \ell\in [0, \ell_0],\ \Proba{\alpha_{R_0}(u)<\ell}\leq \varepsilon^{2n_0}.
    \end{equation}
    In the following of the proof we fix some $\ell>0$ (depending on $\varepsilon$) such that $\ell\leq \ell_0$ and $\ell\leq \frac{h_0}{R_0}$.
    
    Let $A$ be a finite subset of $\Z$ (which we interpret as a subset of $\Z\times \{0\}$ by the natural inclusion $i\mapsto (i,0)$). As mentioned earlier, by the definition of $\alpha$ and $\alpha_{R_0}$, we always have the implication
    $$\alpha(i)<\ell \Rightarrow \alpha_{R_0}(i)<\ell.$$
    Therefore, we obtain the following upper-bound
    \begin{equation}
        \label{a4:eq:alpha_to_alpha_R_0}
        \mathbb{P}(\forall i \in A,\ \alpha(i)<\ell)\leq \mathbb{P}(\forall i \in A,\ \alpha_{R_0}(i)<\ell).
    \end{equation}
    By Lemma \ref{a4:lemma:decomp_R0_connected}, we may write
    $$A = A_1 \sqcup A_2\sqcup \dots \sqcup A_r,$$
    where $(A_k)_{1\leq k \leq r}$ is a $R_0$-decomposition of $A$.
    We can then write
    \begin{align}
        \label{a4:eq:product_R_0}
        \mathbb{P}(\forall i \in A,\  \alpha_{R_0}(i)<\ell)  &= \mathbb{P}(\forall k\in \llbracket 1,r \rrbracket,\  \forall i\in A_k,\ \alpha_{R_0}(i)<\ell)\nonumber\\
        &= \prod_{k=1}^r \mathbb{P}(\forall i\in A_k,\ \alpha_{R_0}(i)<\ell).
    \end{align}
    Indeed, by the definition of a $R_0$-decomposition, the sets $A_k$ are distant of at least $R_0$, and since $\alpha_{R_0}(i)$ only depends on the collection $(Y_j)_{i\leq j\leq i+R_0}$ we see that the restrictions of $\alpha_{R_0}$ to the different $A_k$ are independent.

    We may therefore focus on the case of a set $A\subset \Z$ which is finite and $R_0$-connected. Let $A$ be such a set. We want to prove the following.
    \begin{equation}
        \mathbb{P}(\forall i\in A,\ \alpha_{R_0}(i)<\ell) \leq \varepsilon^{|A|}.
    \end{equation}
    Let us write $A=\{a_1,\dots,a_n\}$ with $a_1<a_2<\dots<a_n$ and where $n=|A|\geq 1$ (in the case $n=0$ there is nothing to prove).
    We distinguish two cases.
    If $n\leq 2n_0$, we may take some $i_0\in A$ (arbitrarily chosen) and write
    \begin{equation}
        \mathbb{P}(\forall i\in A,\ \alpha_{R_0}(i)<\ell) \leq \mathbb{P}(\alpha_{R_0}(i_0)<\ell) \leq \varepsilon^{2n_0}\leq \varepsilon^n.
    \end{equation}
    We now consider the case where $n>2n_0$. First, observe that since $A$ is $R_0$-compact, then
    $$\alpha_{R_0}(a_1)<\ell \Rightarrow Y_{a_1}+\ell(a_2-a_1)> Y_{a_2}.$$
    Moreover, since we assume $\ell\geq 0$ and $a_2-a_1\leq R_0$, and $R_0\ell\leq h_0$ we have
    $$\alpha_{R_0}(a_1)<\ell \Rightarrow Y_{a_1} >_{h_0} Y_{a_2},$$
    where we recall that $Y_{a_1}>_{h_0}Y_{a_2}$ stands for $Y_{a_1}+h_0>Y_{a_2}.$
    We may therefore write
    \begin{align*}
        \mathbb{P}(\forall i\in A,\ \alpha_{R_0}(i)<\ell) & \leq \mathbb{P}(Y_{a_1}>_{h_0}Y_{a_2}>_{h_0} \dots >_{h_0} Y_{a_n}).
    \end{align*}
    This probability is the same as
    \begin{equation}
        \Proba{\forall i\in \llbracket 1, n-1\rrbracket,\ Y_i>_{h_0}Y_{i+1}}.
    \end{equation}
    We will do an upper-bound on  this probability by forgetting one inequality every $n_0$ and then use independence. More precisely
    \begin{align*}
       &\Proba{\forall i\in \llbracket 1, n\rrbracket,\ Y_i>_{h_0}Y_{i+1}} \\ \leq \quad &
       \Proba{\forall j\in \llbracket 1, \lfloor \frac{n}{n_0}\rfloor \rrbracket,\ \forall i \in \llbracket (j-1)n_0+1, jn_0-1\rrbracket,\ Y_i>_{h_0} Y_{i+1}} \\
         \leq\quad & \mathbb{P}(Y_1>_{h_0}Y_2>_{h_0} \dots >_{h_0} Y_{n_0})^{\lfloor \frac{n}{n_0}\rfloor} \\
        \leq \quad&\varepsilon^{2n_0 \lfloor \frac{n}{n_0}\rfloor} \\ \leq \quad & \varepsilon^{2n - 2n_0} \\
        \leq \quad &  \varepsilon^n.
    \end{align*}
    To go from the second to the third line, we use independence and the fact that by stationarity all probabilities are the same. From the third to fourth line we use \eqref{a4:eq:f1}, in the last two lines we use that $\varepsilon\in ]0,1[$, $\lfloor x \rfloor \geq x-1$ and our assumption $n> 2n_0$.
    
    By \eqref{a4:eq:alpha_to_alpha_R_0}, we have thus proven Proposition \ref{a4:prop:regime_ll_1} when $A\subset \Z\times\{0\}$ is finite and $R_0$-connected. By \eqref{a4:eq:product_R_0} and \eqref{a4:eq:alpha_to_alpha_R_0}, Proposition \ref{a4:prop:regime_ll_1} actually holds for any $A\subset \Z\times \{0\}$ a finite subset. To deduce the result for any $A\subset \Z^2$ finite, we may proceed as in the proof of Proposition \ref{a4:prop:regime_gg_1} writing $A=\sqcup_{i\in \Z} A\cap (\Z\times \{i\})$ and using the fact that the restrictions of $\alpha$ to $\Z\times \{i\}$ are independent and identically distributed when $i$ ranges over $\Z$. This concludes the proof of Proposition \ref{a4:prop:regime_ll_1}. 
\end{proof}

\appendix
\section{Recovering $X$ from $\alpha$}
\label{a4:sec:appendixA}
In this appendix we consider the problem of understanding how much information about $X$ does the field $\alpha$ really contain. Indeed we have seen the problem in the following way, we start from a field $X : \Z^2 \to \R$ and we compute a new field $\alpha : \Z^2 \to \overline{\R}$ by the formula
\begin{equation}
    \forall u\in \Z^2,\ \alpha(u) := \sup_{r\geq 1}\frac{X(u+re_1)-X(u)}{r}.
\end{equation}
That is, we may write $\alpha = \Phi(X)$ where $\Phi$ is a deterministic application
\begin{equation}
    \Phi : \R^{\Z^2}\to \overline{\R}^{\Z^2}.
\end{equation}
We wonder if it is possible to build another deterministic application $\Psi : \overline{\R}^{\Z^2} \to \R^{\Z^2}$ such that $\Psi(\Phi(X))=X.$
We will work under the following assumption.
\begin{assumption}
\label{a4:a:a4}
The measure $\mu$ admits a finite first moment, that is
\begin{equation}
    \int_\R |x|\mu(dx)<\infty.
\end{equation}
\end{assumption}
Our main result is the following.
\begin{theorem}
\label{a4:thm:recovery}
There exists a deterministic application $\Psi : \overline{\R}^{\Z^2} \to \R^{\Z^2}$ such that the following holds. Let $X: \Z^2 \to \R$ and $\mu$ be such that Assumptions \ref{a4:a:a1} and \ref{a4:a:a4} hold, then
\begin{equation}
    \Proba{\Psi(\Phi(X))=X-\overline{\mu}\mathds{1}}=1,
\end{equation}
where $\overline{\mu}\in \R$ denotes the first moment of $\mu$ and $\mathds{1}\in \R^{\Z^2}$ has all coordinates equal to $1$.
\end{theorem}
This theorem ensures that under our assumptions it is possible to recover $X$ from $\alpha$ up to an extra additive constant. Moreover if one has access to the first moment of $\mu$ then complete recovery is almost surely possible.
We comment that it is impossible to have an inversion in a classical way (even up to a constant). In fact consider the following counter example. Consider two sequences $(x_i)_{i\in \Z}$ and $(y_i)_{i\in \Z}$ such that,
$$\forall i\in \Z,\  x_i=0 \text{ and }y_i = \begin{cases}
    0 &\text{ if }i\geq 0 \\
    1 &\text{ if }i < 0
\end{cases}.$$ Then if one define $X,Y\in \R^{\Z^2}$ by $X(i,j):=x_i$ and $Y(i,j):=y_i$ we see that $X$ and  $Y$ do not differ by an additive constant but $\Phi(X)=\Phi(Y).$
In order to prove Theorem \ref{a4:thm:recovery} we introduce some notations.

\begin{definition}
    Let $X\in \R^{\Z^2}$ and $j\in \Z^2$. For $u\in \Z$ we denote by $\alpha_j(u)$ the quantity
    $$\alpha_j(u):= \alpha(u,j)=\sup_{r\geq 1}\frac{X(u+r,j)-X(u,j)}{r}.$$
    
    Let $u,v\in \Z$ be such that $u<v$, we denote by $\tau^X_j(u,v)$ the quantity
    \begin{equation}
        \tau^X_j(u,v):= \frac{X(v,j)-X(u,j)}{v-u}.
    \end{equation}
\end{definition}

\begin{claim}
\label{claim:slope}
Let $X\in \R^{\Z^2}$, $j\in \Z$ and $u,v,w$ be three points of $\Z$ such that $u<v<w$, then
\begin{equation}
    (\tau^X_j(u,w)-\tau^X_j(u,v))(\tau^X_j(u,w)-\tau^X_j(v,w))\leq 0.
\end{equation}
Moreover, this product is zero only if both factors are equal to zero.
\end{claim}
\begin{proof}
    The proof is a direct consequence of the fact that $\tau_j^X(u,w)$ can be written as a mean between $\tau_j^X(u,v)$ and $\tau_j^X(v,w).$
    Indeed,
    $$\tau_j^X(u,w) = \frac{w-v}{w-u}\tau_j^X(v,w)+\frac{v-u}{w-u}\tau_j^X(u,v).$$
\end{proof}

We now prove a technical lemma.
\begin{lemma}
\label{a4:lemma:technical_reconstruction}
Assume that $X : \Z^2 \to \R$ and $\mu$ are such that Assumptions \ref{a4:a:a1} and \ref{a4:a:a4} hold. Let $j\in \Z$. Then, almost surely, the following statements are verified,
\begin{enumerate}
    \item For all $u\in \Z$ the set $M(u):=\left\{v\in \Z\ |\ v>u \text{ and }\alpha_j(u) = \tau^X_j(u,v)\right\}$ is non empty and totally ordered. We denote by $T(u)\in \Z$ its minimum.
    \item For all $u\in \Z$, we have the following identity, $$T(u)=\min\{v\in \Z\ |\ v>u\text{ and }\alpha_j(v)\leq \alpha_j(u)\}.$$
    \item For any $u\in \Z$ and $v\in \Z$ such that $u<v<T(u)$ then $T(v)\leq T(u).$
\end{enumerate}
\end{lemma}
\begin{proof}
    We start with the first point of the lemma. Let $u\in \Z$. Let $N\in \N$. For $r\in \N^*$ we have by stationarity
    \begin{align*}
            \Proba{\frac{X(u+r,j)-X(u,j)}{r}\geq 2^{-N}} \leq 2\Proba{2^{N+1}|X(0,0)|\geq r}
    \end{align*}
    Since $\mathbb{E}[|X(0,0)|]<\infty$, then $\sum_{r\in \N^*}\Proba{2^{N+1}|X(0,0)|\geq r}<\infty$ and Borel-Cantelli Lemma implies that almost surely there exists a finite $r_N$ depending on $N$ such that for all $r\geq r_N$, $\frac{X(u+r,j)-X(u,j)}{r}< 2^{-N}$. Since there are countably many $N\in \N$ we may assume that almost surely all $r_N$ are finite.
    
    Now denote by $\mathcal{A}_u$ the event
    $$\mathcal{A}_u := \{\exists r\in \N\ |\ X(u+r,j)>X(u,j)\}.$$
    On the event $\mathcal{A}_u$, let $r_0\in \N$ be such that $X(u+r_0,j)>X(u,j)$. We denote by $\varepsilon_0 := \frac{X(u+r_0,j)-X(u,j)}{r_0}>0.$ Let $N_0\in \N$ be such that $2^{-N_0}<\varepsilon_0$ then by our previous observation we see that eventually for $r\geq r_{N_0}$ we have $\frac{X(u+r,j)-X(u,j)}{r}<\frac{X(u+r_0,j)-X(u,j)}{r_0}$, this proves that $$\alpha_j(u) = \max_{1\leq r \leq r_{N_0}}\frac{X(u+r,j)-X(u,j)}{r}.$$
    which in turns proves the first statement of the lemma on the event $\mathcal{A}_u$.
    On the event $\mathcal{A}_u^c$, then for all $r\in \N$ we have $X(u+r,j)\leq X(u,j)$, this implies $\alpha_j(u)\leq 0.$ We now show that almost surely there exists $r\geq 1$ such that $X(u+r,j)\geq X(u,j)$ which proves $0=\alpha_j(u)=\frac{X(u+r,j)-X(u,j)}{r}$ and concludes the proof of the first element of the lemma. First, we write
    \begin{align*}
        \Proba{\forall r\in \N^*,\ X(u+r,j)<X(u,j)}=\int_{\R}\prod_{r=1}^\infty\mu(]-\infty,x[)\mu(dx).
    \end{align*}
    This comes from the fact that by Assumption \ref{a4:a:a1}, all $X(v)$ are independent and distributed according to the same probability measure $\mu.$ Denote by $I$ the set
    $$I := \{x\in \R\ |\ \mu(]-\infty,x[)=1\}.$$
    It turns out that $I$ is a measurable set. In fact, either $I=\emptyset$ or $I$ is an interval of the form $]a,\infty[$ or $[a,\infty[.$ The above computation can be written as
    \begin{align*}
        \Proba{\forall r\in \N^*,\ X(u+r,j)<X(u,j)}=\mu(I).
    \end{align*}
    If $I=\emptyset$, then the probability is $0$ and we are done. Otherwise, let $a=\inf I\in \R$, let $(a_n)_n$ be a decreasing sequence of element in $I$ such that $a_n \to a.$ We have
    $$\mu(]-\infty,a])=\mu(\bigcap_{n\in \N}]-\infty,a_n[)=1.$$
    There are two cases. If $\mu(\{a\})>0$, then $\mu(]-\infty,a[)<1$ so that $a\not\in I$, and $I=]a,\infty[$. Therefore $\mu(I)=\mu(]a,\infty[)=1-\mu(]-\infty,a])=0.$ Otherwise, if $\mu(\{a\})=0$, then $I=[a,\infty[$ and $\mu(I)=1-\mu(]-\infty,a])+\mu(\{a\})=0.$
    In all cases we have $\mu(I)=0$, this shows that with probability $1$ then there exists $r\in \N^*$ such that $X(u+r,j)\geq X(u)$ and this implies the conclusion as mentioned earlier.
    
    For the second item of the lemma. We first show that $\alpha_j(T(u))\leq \alpha_j(u)$. In fact, let $v>T(u)$ be such that $\alpha_j(T(u))=\tau^X_j(T(u),v).$ By definition of $\alpha_j(u)$ and $T(u)$ we have
    $$\alpha_j(u)=\tau^X_j(u,T(u))\geq \tau^X_j(u,v).$$
    Therefore, by Fact \ref{claim:slope}, we have $$\tau^X_j(T(u),v)\leq \tau^X_j(u,v).$$
    This entails $$\alpha_j(T(u))\leq \alpha_j(u).$$
    We now show that for any $v\in \Z$ such that $u<v<T(u)$ then $\alpha_j(v)>\alpha_j(u)$. Indeed, for such $v$, then by definition of $T(u)$ we have
    $$\tau^X_j(u,v)<\tau^X_j(u,T(u)).$$
    Therefore, by Fact \ref{claim:slope} we have
    $$\tau^X_j(v,T(u))>\tau^X_j(u,T(u))=\alpha_j(u).$$
    This implies $\alpha_j(v)>\alpha_j(u).$
    
    For the third element of the lemma, assume by contradiction that we have $u,v\in \Z$ such that $u<v<T(u)$ and verifying $T(v)>T(u).$
    Then by definition of $T(v)$ we have $$\tau^X_j(v,T(u))<\tau^X_j(v,T(v))=\alpha_j(v)$$
    Therefore, by Fact \ref{claim:slope} we have
    $$\tau^X_j(T(u),T(v))>\tau^X_j(v,T(v))=\alpha_j(v).$$
    This implies $\alpha_j(T(u))>\alpha_j(v)$ which is in contradiction with the second item of the lemma as we would have $\alpha_j(u)\geq \alpha_j(T(u))>\alpha_j(v)>\alpha_j(u)$.
\end{proof}

We now provide the proof of Theorem \ref{a4:thm:recovery}
\begin{proof}[Proof of Theorem \ref{a4:thm:recovery}]
    We will build a function $\Psi : \overline{\R}^{\Z^2}\to \R^{\Z^2}$ such that almost surely $\Psi(\Phi(X))=X-\overline{\mu}\mathds{1}.$ In order to do so, we first build a function $\Psi_0 : \overline{\R}^{\Z^2}\to \R^{\Z^2}$ such that almost surely, $$\Psi_0(\Phi(X))=(X(i,j)-X(0,j))_{i,j\in \Z}.$$
    Let us give some element $\alpha\in \overline{\R}^{\Z^2}.$
    We we build $Y=\Psi_0(\alpha)$ algorithmically. During the various steps of the algorithm we will need to make some assumptions on $\alpha$ (that are verified almost surely when $\alpha=\Phi(X)$). Therefore, to define $\Psi_0$ and $\Psi$ as functions on the whole domain $\overline{\R}^{\Z^2}$ we arbitrary choose an element $\dagger$ of $\R^{\Z^2}$ such that $\Psi_0(\alpha)=\Psi(\alpha)=\dagger$ if $\alpha$ is a "bad configuration" (which almost surely will not happen).
    First, if there exists $u\in \Z^2$ such that $\alpha(u)\not\in \R$ we define $\Psi_0(\alpha)=\dagger.$
    We now assume that for all $u\in \Z^2$, $\alpha(u)\in \R.$
    Let $j\in \Z$ we will first define $Y$ on $\Z\times \{j\}$. Let $u\in \Z$. As in Lemma \ref{a4:lemma:technical_reconstruction} we denote by $T(u)$ the index
    $$T(u):= \inf\{v\in \Z\ |\ v>u\text{ and }\alpha_j(v)\leq \alpha_j(u)\},$$
    with the convention $\inf \emptyset =\infty.$ If there exists $u\in \Z$ such that $T(u)=\infty$ then we set $\Psi_0(\alpha)=\dagger.$
    We now assume that all $T(u)$ are finite. This will almost surely be the case by Lemma \ref{a4:lemma:technical_reconstruction}.
    Let $(x_n)_{n\geq 0}\in \Z^\N$ the sequence defined by
    $$x_0=0 \quad \quad \forall n\in \N,\ x_{n+1} := T(x_n).$$
     We now define $\Psi_0(\alpha)$ on the collection $(x_n,j)_{n\geq 0}$ by
     $$Y(x_0,j):=0$$
     $$\forall n\in \N,\ Y(x_{n+1},j) := Y(x_{n},j)+(x_{n+1}-x_{n})\alpha_j(x_n).$$
    We proceed to extend this definition on $\N\times \{j\}.$
    Let $n\in \N$, such that $x_n+1<x_{n+1}$ we define $Y$ on $\llbracket x_n+1, x_{n+1}\rrbracket\times \{j\}$ by induction. $Y(x_{n+1},j)$ was already defined earlier. If we assume that $Y$ is defined on $\llbracket x_{n+1}-r, x_{n+1}\rrbracket\times \{j\}$ for some $0\leq r \leq x_{n+1}-x_n-2.$ Then let $x:= x_{n+1}-(r+1)\in \Z$ we define $Y(x,j)$ by the following rule.
    \begin{itemize}
        \item If $T(x)\leq x_{n+1}$ we set $Y(x,j):= Y(T(x),j)-\alpha(x)(T(x)-x).$
        \item Otherwise we set $\Psi_0(\alpha)=\dagger.$
    \end{itemize}
    By the third element of Lemma \ref{a4:lemma:technical_reconstruction} we see that if $\alpha=\Phi(X)$ then, almost surely, we will have $T(x)\leq x_{n+1}$ and therefore we will not set $\Psi_0(\alpha)=\dagger.$
    
    Once we have constructed $Y$ on $\N\times \{j\}$ by the above procedure we may proceed to define $Y$ on $\Z\times \{j\}.$ 
    We proceed by induction, assume that $Y$ is constructed on $\llbracket -r, \infty \llbracket\times \{j\}$ for some $r\geq 0$ we extend this construction to $\llbracket -(r+1),\infty\llbracket\times \{j\}$ by setting
    $$Y(-(r+1),j) := Y(T(-(r+1)),j)-\alpha_j(-(r+1))(T(-(r+1))+(r+1)).$$

    By repeating this for all $j\in \Z$ we build the process $Y$ on $\Z\times \Z.$ This process was built so that $Y(0,j)=0$ and $Y(T(u),j)-Y(u,j)=\alpha_j(u)(T(u)-u)$ for all $u,j\in \Z$. It is therefore straightforward to verify by Lemma \ref{a4:lemma:technical_reconstruction} that almost surely we have $$Y=\Psi_0(\Phi(X))=(X(i,j)-X(0,j))_{i,j\in \Z^2}.$$
    
    Once we have computed $Y=\Psi_0(\alpha)$, we define $\Psi(\alpha)$ in the following way,
    $$\forall (i,j)\in \Z^2,\ \Psi(\alpha)(i,j):=Y(i,j) - \lim_{N\to \infty}\frac{1}{N}\sum_{r=1}^N Y(r,j).$$
    As usual, if the limit above does not make sense then we set $\Psi(\alpha)=\dagger.$
    When $\alpha=\Phi(X)$ then, almost surely, we have $Y(r,j)=X(r,j)-X(0,j)$ and by the law of large numbers the limit exists,
    $$\frac{1}{N}\sum_{r=1}^N Y(r,j) \xrightarrow[n\to \infty]{} \overline{\mu}-X(0,j).$$
    This shows that almost surely
    $$\Psi(\Phi(X))=X-\overline{\mu}\mathds{1},$$
    which concludes the proof of Theorem \ref{a4:thm:recovery}.
\end{proof}

\bibliography{biblio.bib}

David, Vernotte

Institut Fourier, UMR 5582, Laboratoire de MathématiquesUniversité Grenoble Alpes, CS 40700, 38058 Grenoble cedex 9, France
\end{document}